\documentclass[12pt]{article}
\usepackage{epsf}
\usepackage{amsbsy,amsmath}
\usepackage{amsfonts}
\usepackage{amssymb}
\usepackage{eucal}
\usepackage{graphics,mathrsfs}
\usepackage{amsthm}
\usepackage{secdot}
\newtheorem{theorem}{Theorem}[section]

\newtheorem{corollary}[theorem]{Corollary}
\theoremstyle{definition}
\newtheorem{definition}[theorem]{Definition}
\newtheorem{example}[theorem]{Example}

\theoremstyle{remark}
\newtheorem{remark}[theorem]{Remark}
\numberwithin{equation}{section}

\numberwithin{equation}{section}

\newsavebox{\savepar}

\begin{document}
\title{\bf On the properties of Ascent and Descent of Composition operator on Orlicz spaces}
\author{Ratan Kumar Giri \footnote{Corresponding author}\,, Shesadev Pradhan \\
 {\it{\small Dept. of Mathematics, National Institute of Technology, Rourkela, Odisha, India }}\\
{\it{\small e-mail: giri90ratan@gmail.com, pradhansh@nitrkl.ac.in}}
 }

\date{}
\maketitle
\begin{abstract}
Here, the composition operators on Orlicz spaces
with finite ascent and descent as well as infinite ascent and
descent are characterized.
\end{abstract}

\section{Introduction and preliminaries}
Before going to start, let us recall the notion of ascent and
descent of an operator on an arbitrary vector space $X$. So if $T: X
\rightarrow X$ is an operator on $X$, then $N(T)$ and $R(T)$ denotes
the null space and range space of $T$ respectively, that is
$$N(T)=\{x \in X: T(x)=0\}\,\,\,\,\,\mbox{and}\,\,R(T)= \{T(x): x \in X\}$$
The null space of $T^k$ is a $T$-invariant subspace of $X$, that is,
$T(N(T^k))\subseteq N(T^k)$ for every positive integer $k$.  Indeed,
if $x\in N(T^k)$ then $T^k(x)=0$ and therefore,
$T^k(T(x))=T(T^k(x))=0$, i.e., $T(x)\in N(T^k)$. We also have the
following subspace inclusions:
$$N(T)\subseteq N(T^2) \subseteq N(T^3)\subseteq \cdots$$
The range $R(T^k)=T^k(X)$ of each operator $T^k$ is clearly another
$T$-invariant subspace of $X$. Moreover:
$$R(T)\supseteq R(T^2) \supseteq R(T^3)\supseteq \cdots$$
Following definitions and well known results are relevant to our
context (\cite{func11}, \cite{func}, \cite{func12});
\begin{theorem}
 For an operator $T: X \rightarrow X$ on a vector space we have have following:
 \begin{enumerate}
  \item If $N(T^k)= N(T^{k+1})$ for some $k$, then $N(T^n)= N(T^k)$ for all $n\geq k$.
  \item If $R(T^k)= R(T^{k+1})$ for some $k$, then $R(T^n)= R(T^k)$ for all $n\geq k$.
 \end{enumerate}
\end{theorem}
We now introduce ascent and descent of an operator.
\begin{definition}
Let $T:X\rightarrow X$ be an operator on a vector space.
\begin{enumerate}
 \item The ascent $\alpha(T)$ of $T$ is the smallest natural number $k$ such that $N(T^k)= N(T^{k+1})$. If there is no $k\in \mathbb{N}$ such that
 $N(T^k)= N(T^{k+1})$, then we say that ascent $\alpha(T)$ is infinite.
 \item Similarly, the descent $\delta(T)$ of $T$ is the smallest natural number $k$ such that $R(T^k)= R(T^{k+1})$. If there is no $k\in \mathbb{N}$ such that
 $N(T^k)= N(T^{k+1})$, then we say that descent $\delta(T)$ is infinite.
\end{enumerate}
\end{definition}
It turns out that if ascent and descent of an operator is finite
then they are equal. This useful result stated next:
\begin{theorem}
 If an operator $T: X \rightarrow X$ on a vector space has finite ascent and descent, then they must be coincide i.e., $\alpha(T)= \delta(T)=p <\infty$,
 and $$X= N(T^p)\oplus R(T^p)$$
 Moreover, if $x$ is a Banach space and $T$ is linear, then $R(T^p)$ is a closed subspace.
\end{theorem}
Let $T: X \rightarrow X$ be a bounded operator on the Banach space
$X$. A pair $(V, W)$ of closed subspaces of $X$ is said to be a
reducing pair of the operator $T$ if $X= V\oplus W$. Now we have the
following result:
\begin{theorem}
 A bounded operator $T:X \rightarrow X$ on a Banach space $X$ has finite ascent and descent if and only if $T$ has a reducing pair of closed subspaces $(V, W)$
 such that the operator $T: V \rightarrow V$ is nilpotent and $T: W \rightarrow W $ is invertible.\\
 Moreover, if $p=\alpha(T)= \delta(T) <\infty$,  then the pair $(V, W)$ where $V= N(T^p)$ and $W= R(T^p)$ is the only reducing pair.
\end{theorem}
\par F. Riesz in \cite{func13} introduced the concept of ascent and descent for a linear operator in a connection
with his investigation of compact linear operators. Also the study
of ascent and descent has been done as a part of spectral properties
of an operator (\cite{func14}, \cite{func15}). Since the composition
operator provide the diverse and illuminating example of operators
which leads to study useful insight into structure theory of
operators, it is desirable to study ascent and descent of these
operators. In this paper, we study the ascent $\alpha(T)$ and
descent $\delta(T)$, where $T$ is a composition operator on Orlicz
Spaces.
\section{Composition operator on Orlicz spaces}
Let $(\Omega,\Sigma,\mu)$ be a $\sigma$-finite complete measure
space, where $\Sigma$ is a $\sigma$-algebra of subsets of an
arbitrary set $\Omega$ and $\mu$ is a non-negative measure on
$\Sigma$. Let $\phi :[0,\infty)\rightarrow[0,\infty)$ be
non-decreasing continuous convex function such that
%\begin{itemize}
 $\phi(x)=0$ if and only if $x=0$ with
 $\displaystyle {\lim_{x\rightarrow 0}\phi(x)/x=0}$ and
$\displaystyle {\lim_{x\rightarrow \infty}\phi(x)/x=\infty}.$
%\end{itemize}
Such a function $\phi$ is known as an Orlicz function. Let
$L^{0}(\Omega)$ be denote the linear space of all equivalence
classes of $\Sigma$-measurable functions on $\Omega$, where we
identify any two functions are equal in the sense of $\mu$-almost
everywhere on $\Omega$. Then the functional $I_\phi :L^{0}(\Omega)
\rightarrow [0,\infty]$, defined by
$$I_\phi(f)=\int_\Omega\Phi(f(t))d\mu(t)$$ where $f\in
L^{0}(\Omega)$, is a pseudomodular \cite{Ames97}, which is also
defined as a modular in \cite{Ames98}. Let $L^\phi (\Omega)$ be the
set of all measurable function such that $\int_{\Omega}\phi(\alpha
|f|)d\mu < \infty$ for some $\alpha>0$. The space $L^{\phi}(\Omega)$
is called as Orlicz space and it is a Banach space with two norms:
the Luxemberg norm \cite{Ames99}, defined as
$$ ||f||_\phi=\inf \left \{k>0:I_\phi\left(\frac{|f|}{k}\right)\leq 1\right\}$$ and the
Orlicz norm in the Amemiya form  \cite{Ames82,Ames94} is given as
$$||f||^{0}_\phi = \displaystyle {\inf _{k>0}(1+I_\phi(kf))/k}.$$
 Note that the
equality of the Orlicz norm and the Amemiya norm was proved in
\cite{Ames82}. If $\phi(x)= x^p$, $1<p<\infty$, then
$L^{\phi}(\Omega)= L^p$, the well known Banach space of
$p$-integrable function on $\Omega$ with $||f||_\phi=
\left(\frac{1}{p}\right)^{\frac{1}{p}} ||f||_p$ (\cite{Ames99}). It
is well known that $||f||_\phi \leq ||f||^{0}_\phi\leq 2||f||_\phi$
and $||f||_\phi \leq 1$ if and only if $I_\phi(f)\leq 1$
(\cite{Ames77,Ames80}). Moreover, if $A\in \Sigma$ and
$0<\mu(A)<\infty$, then
$||\chi_A||_\phi=\frac{1}{\phi^{-1}(\frac{1}{\mu(A)})},$ where
$\chi_A$ is the characteristic function on $A$ (page no 78,
\cite{Ames94}). For more literature concerning Orlicz spaces, we
refer to Kufener, John and Fucik \cite{Ames79},  Musielak
\cite{Ames72} and Rao \cite{Ames94}.
\par Let $\tau:\Omega\rightarrow\Omega$ be a measurable transformation,
that is, $\tau^{-1}(A)\in\Sigma$ for any $A\in\Sigma$. If
$\mu(\tau^{-1}(A))=0$ for any $A\in\Sigma$ with $\mu(A)=0$, then
$\tau$ is called as nonsingular. This condition implies that the
measure $\mu\circ\tau^{-1}$, defined by $\mu\circ\tau^{-1}(A)
:=\mu(\tau^{-1}(A))$ for $A\in\Sigma$, is
 absolutely continuous w.r.t $\mu$ ($\mu\circ\tau^{-1}\ll\mu$). Then
the Radon-Nikodym theorem implies that there exist a non-negative
locally integrable function $f_{\tau}(x)$ on $\Omega$ such that
$$\mu\circ\tau^{-1}(A)=\int_{A}f_{\tau}(x)d\mu(t)\,\,\,\quad\mbox{for}
\,\, A\in\Sigma.$$ Any nonsingular measurable transformation $\tau$
induces a linear operator (Composition operator)  $C_{\tau}$ from
$L^{0}(\Omega)$ into itself which is defined as
$$C_{\tau}f(t)= f(\tau(t)),\quad t\in\Omega,\quad f\in L^{0}(\Omega).$$
Here, the non-singularity of $\tau$ guarantees that the operator
$C_{\tau}$ is well defined. Now, if the linear operator $C_\tau$
maps an Orlicz space $L^\phi(\Omega)$ into itself, then we call
$C_\tau$ is a composition operator in $L^\phi(\Omega)$.
\par An Orlicz function $\phi$ is said to be satisfied the $\Delta_2$ condition if there exists a positive constant $K$ such
that $\phi(2x)\leq K \phi(x)$ for all $x> 0$ (\cite{Ames80}).
In \cite{blum74} the necessary and sufficient condition about the boundedness and compactness of
composition operators on Orlicz spaces are described. Regarding the
boundedness of the composition operator $C_\tau$ from an Orlicz
space $L^\phi(\Omega)$ into itself, we have the following theorem
(Theorem 2.2, \cite{blum74}).
\begin{theorem}
The composition operator $C_\tau$ is bounded from an Orlicz space
$L^\phi(\Omega)$ into itself if $\mu(\tau^{-1}(A))\leq K\mu(A)$
holds for some $K>0$ and for all $A\in \Sigma$ with $\mu(A)<\infty$
and also converse holds when the Orlicz function $\phi $ satisfies
$\Delta_2$ condition for all $x>0$.
\end{theorem}
Through out this paper, we assume that the composition operator is
continuous. In \cite{inte11}, kernel of the Composition operator
$C_\tau$ is obtained. It is shown that $N(C_\tau)=
L^\phi(\Omega_\circ)$, where $\Omega_\circ =\{x\in \Omega:
f_{\tau}(x) =0\}$ and $L^\phi(\Omega_\circ)=\{f\in L^\phi(\Omega):
f(x)=0~~ \mbox{for}~~ x\in \Omega \setminus \Omega_\circ\}$. For
systematic study on composition operators on different spaces we
refer to \cite{Ames102}, \cite{Ames81} and \cite{Ames101}.
\par Now, we will characterized the composition operator
on Orlicz spaces $L^\phi(\Omega)$ with finite ascent and descent as
well as infinite ascent and descent.
\section{Main Results}
Let us consider $ \tau: \Omega \rightarrow \Omega$ be a nonsingular
measurable transformation. Now $\tau$ is a non singular measurable
transformation implies that $\tau^k$ is also non singular measurable
transformation for every $k\geq2$ with respect to the measure $\mu$.
Hence $\tau^k$ is also induces a composition operator $C_{\tau^k}$.
Note that for every measurable function $f$, $C^k_\tau(f)= f\circ
\tau^k = C_{\tau^k}(f)$. Also we have
$$\cdot \cdot\cdot \ll\mu\circ \tau^{-(k+1)}\ll\mu\circ \tau^{-k}\ll \cdot \cdot \cdot \ll\mu\circ \tau^{-1}\ll\mu .$$
Take $\mu \circ \tau^{-k}= \mu_k$. Then by Radon-Nikodym theorem,
there exists a non-negative locally integrable function $f_{\tau^k}$
on $\Omega$ so that the measure $\mu_k$ can be represented as $$
\mu_k(A) = \int_A f_{\tau^{k}}(x) d\mu(x),
\,\,\,\mbox{for~~all~~$A\in\Sigma$.}$$ The function $f_{\tau^k}$ is
known as the Radon-Nikodym derivative of the measure $\mu_k$ with
respect to the measure $\mu$.  Now the following theorem
characterized the composition operators with ascent $k$:
\begin{theorem}
The composition operator on Orlicz space $L^\phi(\Omega)$ has ascent
$k\geq 1$ if and only if $k$ is the first positive integer such that
the measures $\mu_k$ and $\mu_{k+1}$ are equivalent.
\end{theorem}
\begin{proof}
Suppose that $\mu_k$ and $\mu_{k+1}$ are equivalent. Then $\mu_{k+1}
\ll\mu_k\ll\mu_{k+1}$. Since $\mu_k\ll\mu_{k+1}\ll\mu$, hence the
chain rule of Radon-Nikodym derivative implies that
\begin{align}
\frac{d\mu_k}{d\mu}(x) & = \frac{d \mu_k}{d \mu_{k+1}}(x) \cdot
\frac{d\mu_{k+1}}{d\mu}(x)\\
\Rightarrow f_{\tau^k}(x) & =\frac{d \mu_k}{d \mu_{k+1}}(x)\cdot
f_{\tau^{k+1}}(x) \label{eq:1}
\end{align}
Similarly, $\mu_{k+1}\ll\mu_k\ll\mu$ implies that
\begin{eqnarray}
f_{\tau^{k+1}}(x)= \frac{d\mu_{k+1}}{d\mu_{k}}(x)\cdot
f_{\tau{k}}(x)\label{eq:2}
\end{eqnarray}
Now kernel of the $C^k_\tau$ given by $N(C^k_\tau) = N(C_{\tau^k})
=L^\phi(\Omega_k)$ where $\Omega_k = \{x\in \Omega:
f_{\tau^k}(x)=0\}$ and $L^\phi(\Omega_k)= \{f\in
\L^\phi(\Omega):f(x)=0~~ \mbox{for}~~ x\in \Omega \setminus
\Omega_k\}$. Similarly,  $N(C^{k+1}_\tau) = L^\phi(\Omega_{k+1})$,
where $\Omega_{k+1} = \{x\in \Omega: f_{\tau^{k+1}}(x)=0\}$. Now
from $\ref{eq:1}$ and $\ref{eq:2}$, it follows that $\Omega_k=
\Omega_{k+1}$. Therefore we have,
$$ N(C^k_\tau)= L^\phi(\Omega_k)= L^\phi(\Omega_{k+1})=
N(C^{k+1}_\tau).$$ Since $k$ is the least hence ,This shows that ascent of $C_\tau$ is $k$.\\
Conversely, suppose that ascent of $C_\tau$ is $k$. Now this implies
that if $N(C^k_\tau)= L^\phi(\Omega_k)$ and
$N(C^{k+1}_\tau)=L^\phi(\Omega_{k+1})$, then
$L^\phi(\Omega_k)=L^\phi(\Omega_{k+1})$. Hence $\Omega_k=
\Omega_{k+1}$ almost everywhere with respect to the measure $\mu$.
So $\Omega_k = \{x\in \Omega: f_{\tau^k}(x)=0\}=\{x\in \Omega:
f_{\tau{k+1}}(x)=0\}$. Now it is known that $\mu_{k+1}\ll\mu_{k}$.
Only to show $\mu_k \ll \mu_{k+1}$. For this let $E\in \Sigma $ such
that $\mu_{k+1}(E)=0$. Now we have the following cases:\\
Case-$1$: when  $E\cap \Omega_k =\emptyset$\\
Then $0= \mu_{k+1}(E)= \int_E f_{\tau^{k+1}}(x)d\mu(x)$ implies that
$\mu(E)=0$ as on $E$, $f_{\tau^{k+1}}(x)>0$. As $\mu_k(E)= \int_E
f_{\tau^k}(x) d\mu(x) $ and $\mu(E)=0$, hence $\mu_k(E)=0$.\\
Case-$2$: when $E\cap \Omega_k\neq \emptyset$\\
Then we have,
\begin{align*}
0 =\mu_{k+1}(E) &  = \int_E
f_{\tau^{k+1}}(x)d\mu(x)\\
& = \int_{E\setminus (E\cap \Omega_k)}f_{\tau^{k+1}}(x)d\mu(x) +
\int_{E\cap \Omega_k}f_{\tau^{k+1}}(x)d\mu(x)\\
& = \int_{E\setminus (E\cap \Omega_k)}f_{\tau^{k+1}}(x)d\mu(x)
\end{align*}
Now this implies that $\mu(E\setminus (E\cap \Omega_k))=0$.
Therefore, in either case $\mu_{k+1}(E)=0$ implies that
$\mu_k(E)=0$. Thus $\mu_{k+1}\ll\mu_k\ll\mu_{k+1}$.
\end{proof}
\begin{corollary}
Ascent of the composition operator $C_\tau$ on Orlicz spaces is
infinite if and only if there does not exist any positive integer
$k$ such that the measures $\mu_k$ and $\mu_{k+1}$ are equivalent.
\end{corollary}
\par We say that a measurable transformation $\tau$ is measure preserving if $\mu(\tau^{-1}(E)) =
\mu(E)$ for all $E\in \Sigma$. Then we have the following results:
\begin{corollary}
\begin{enumerate}
\item  If the measure $\mu$ is measure preserving then the ascent of
the composition operator $C_\tau$ on Orlicz spaces $L^\phi(\Omega)$
is 1.
\item If $\tau$ is a nonsingular surjective measurable transformation such that
$\mu(\tau^{-1}(E))\geq \mu(E)$ for all $E\in \Sigma$, then also
ascent of the composition operator induced by $\tau$ on Orlicz
spaces is 1.
\end{enumerate}
\end{corollary}
Note that for a non-singular surjective measurable transformation
$\tau: \Omega \rightarrow \Omega$, measure of set a $A\subseteq
\Omega$ positive does not always imply that the measure of
$\tau^{-1}(A)$ is also positive. Example of one such measurable
transformation is following:
\begin{example}
Consider unit interval $[0, 1]$ with Lebesgue measure. Let $C$ be
the Cantor set. Map $C$ onto $[0, \frac{1}{2}]$. For example convert
ternary expansion to binary expansion and half it. Let the map be
$S_1$. Next, map $[0,1]\setminus C$ onto $[\frac{1}{2}, 1]$ in a
non-singular way. In fact we can get a one-one bimeasurable map
$S_2$ from $[0,1]\setminus C$ onto $[\frac{1}{2},1]$ such that
measure of $S_2(A)$ equals half measure of $A$, for each $A$ in
$[0,1]\setminus C$. Let $\tau$ be the map which is $S_1$ on $C$ and
$S_2$ on $[0,1]\setminus C$.\\
This will satisfy our requirement. Here $[0,\frac{1}{2}]$ has
positive measure but its inverse has measure zero. However if $B$ is
a set of zero measure then measure of $\tau^{-1}(B)$ is same as
measure of $S^{-1}_2 (B)$ intersected with [1/2,1].
\end{example}

Now the next result gives a necessary and sufficient condition for
infinite ascent of the composition operator $C_\tau$ on Orlicz
spaces $L^\phi(\Omega)$ in terms of range of $\tau$.
\begin{theorem}
 Suppose that in the measure space $\Omega= (\Omega, \Sigma, \mu)$,  $\tau: \Omega \rightarrow \Omega$ is a nonsingular surjective measurable transformation such that if $\mu(A)>0$ then also $\mu(\tau^{-1}(A))>0$, where $A\in \Sigma$. Then ascent of $C_\tau$ on Orlicz space $\L^\phi(\Omega)$ is infinite if and only if there exists a
 sequence of subsets $\{\Omega_k\}$ of $\Omega$ such that for all $k\geq 1$
 \begin{enumerate}
  \item $0<\mu(\Omega_k) < \infty$
  \item $\Omega_k \nsubseteq R(\tau^k)$ but $\Omega_k \subseteq R(\tau^{k-1})$
  \item $\mu(\Omega_k \cap R(\tau^k)) = 0$ and $\mu(\Omega_k\cap \Omega_{k+1})=0$.
 \end{enumerate}
\end{theorem}
\begin{proof}
 Assume that ascent of $C_\tau$ is infinite. Then $N(C^{k}_\tau) \neq N(C^{k+1}_\tau)$ for every $k\geq 0$. For $k=0$, this implies that there exists $f\neq 0$ a.e. in
 $\L^\phi(\Omega)$ such that $f\in N(C_\tau)$ i.e., $f\circ\tau=0$ a.e. . Take $\Omega_\circ =\{x\in\Omega: f(x)=0\}$. As $f\neq 0$, the set
 $\Omega \setminus \Omega_\circ= \Omega_1$ (say) is not empty. Since the measure $\mu$ is $\sigma$-finite, hence it has a subset of finite (positive) measure.
 Without loss of generality, we may assume that $0<\mu(\Omega_1)<\infty$.\\
 Claim: $\Omega_1\nsubseteq R(\tau)$\\
If $\Omega_1 \subseteq R(\tau)$, then $\tau(W_1)=\Omega_1$, where
$W_1= \{x\in\Omega : \tau(x)\in \Omega_1\}$. Then for all $x\in
W_1$, $f\circ \tau(x) \neq0$ and by the given conation of $\tau$,
$\mu(W_1)= \mu(\tau^{-1}(\Omega_1))>0$ as $0<\mu(\Omega_1)<\infty$. This shows that $f\circ \tau \neq 0$ a.e. and $f\notin N(C_\tau)$, which is a contradiction.\\
Claim: $\mu(\Omega_1 \cap R(\tau))=0$\\
If $\Omega_1\cap R(\tau)=\emptyset$, then its obvious. If not, then
$\Omega_1\cap R(\tau)= \Omega '_1$ (say) is a proper subset of both
$\Omega_1$ and $R(\tau)$ respectively. Then $\tau(W'_1)= \Omega'_1$,
where $W'_1 = \{x \in \Omega: \tau(x) \in \Omega'_1 \subset
\Omega_1\}$. Now if $\mu(\Omega'_1)>0$, then by the given condition
of $\tau$, implies that $ \mu(W'_1)=\mu(\tau^{-1}(\Omega'_1))>0$.
But for all $x\in W'_1$, $f\circ \tau (x) \neq 0$. This shows that
$f\notin N(C_\tau)$, which is a
contradiction. Therefore $\mu(\Omega'_1)=0$.\\
Now for $k=1$, we have $N(C_\tau)\neq N(C^2_\tau)$. This implies
that there exist $f\neq 0$ a.e. in $L^\phi(\Omega)$ such that
$f\circ \tau^2=0$ a.e. but $f\circ \tau\neq 0$ a.e. . So if
$\Omega'_2=\{x\in \Omega: f\circ \tau(x)\neq 0\}$, then by the
$\sigma$-finiteness of measure $\mu$, we can assume that
$0<\mu(\Omega '_2)<\infty$. Take $\Omega_2=\tau(\Omega '_2)$. Then
$\Omega_2 \in R(\tau)$. As $\tau$ satisfies the given condition as
above and $C_\tau$ is bounded, hence $0<\mu(\Omega_2)<\infty$. Now
$\tau$ is nonsingular measurable transformation imply that $\tau^2$
also. Then by the similar kind of arguments it can be seen that
$\Omega_2
\nsubseteq R(\tau^2)$ and $\mu(\Omega_2\cap R(\tau^2))=0$.\\
Therefore, we have two subsets $\Omega_1$ and $\Omega_2$ of finite
measure with the following properties:
\begin{itemize}
 \item$\Omega_1 \subseteq R(\tau^0) = R(I)$ ( $I$ denote the identity map ) but $\Omega_1\nsubseteq
 R(\tau)$ and $\mu(\Omega_1\cap R(\tau))=0$
 \item  $\Omega_2 \subseteq R(\tau)$ but $\Omega_2\nsubseteq
 R(\tau^2)$ and $\mu(\Omega_2\cap R(\tau^2))=0$
\end{itemize}
Now $\Omega_1 \cap \Omega_2$ is a subset of $\Omega_1\cap R(\tau)$.
As $\mu(\Omega_1\cap R(\tau))=0$, hence $\mu(\Omega_1 \cap
\Omega_2)=0$. Hence by continuing similar process for every $k\geq
2$, we get a sequence $\{\Omega_k\}$ of subsets of finite measure
such that $\Omega_k \nsubseteq R(\tau^k)$ but $\Omega_k \subseteq
R(\tau^{k-1})$ and  $\mu(\Omega_k \cap R(\tau^k)) = 0$ and
$\mu(\Omega_k\cap \Omega_{k+1})=0$.\\
Conversely, suppose that the given conditions are holds. As
$0<\mu(\Omega_k)<\infty $, hence the characteristic function
$\displaystyle{\chi_{\Omega_k}}$ is in Orlicz space
$L^{\phi}(\Omega)$. As $\Omega_k \subseteq R(\tau^{k-1})$ hence
$\tau^{k-1}(W_k)= \Omega_k$, where $W_k=\{x\in \Omega:
\tau^{k-1}(x)\in \Omega_k\}$. Then we have,
\begin{align*}
C^{k-1}_\tau \chi_{\Omega_k}(W_k) & = \chi_{\Omega_k}(\tau^{k-1}(W_k)) \\
& =\chi_{\Omega_k}(\Omega_k) \\
& =1
\end{align*}
This implies that $\chi_{\Omega_k}\notin N(C^{k-1}_\tau)$. But
\begin{align*}
C^k_\tau\chi_{\Omega_k}(\Omega) & =\chi_{\Omega_k}(\tau^k(\Omega))\\
 & =0\,\,\,\,\mbox{a.e.}
\end{align*}
as $\Omega_k \nsubseteq R(\tau^k)= \tau^k(\Omega)$ and
$\mu(\Omega_k\cap R(\tau^k))=0$. Therefore, we have
$\chi_{\Omega_k}\in N(C^k_{\tau})$ but $\chi_{\Omega_k}\notin
N(C^{k-1}_\tau)$. Since $k$ is arbitrary, hence the ascent of
$C_\tau$ is infinite.
\end{proof}
%\begin{corollary}
% If on the measure space $\Omega= (\Omega, \Sigma, \mu)$, $\tau$ is a measure preserving transformation, then the following are equivalent:
% \begin{enumerate}
%  \item  There exists a natural number $k$ (smallest) such that $R(\tau^k)= R(\tau^{k+1})$.
%  \item Ascent of the composition operator on Orlicz spaces is $k$.
%  \item $k$ is the least natural number such that two measures $\mu_k$ and $\mu_{k+1}$ are equivalent.
% \end{enumerate}
%\end{corollary}

\begin{remark}
Here if we take our Orlicz function $\phi$ to be $x^p$ with $\Omega
= \mathbb{N}$ and the measure $\mu$ is the counting measure, then
the Orlicz space $L^\phi(\Omega)$ becomes the well known $l^p$
sequence spaces. Now suppose that $\tau: \mathbb{N}\rightarrow
\mathbb{N}$ is onto. Then the measure $\mu$ satisfy previous
criteria for $\tau$. Hence by the previous theorem we can say that
\it{``ascent of the composition operator $C_\tau$ on $l^p$ sequence
space is infinite if and only if there exists a sequence of
disjoints positive integers $\{n_k\}$ such that $n_k \notin
R(\tau^k)$ but $n_k \in R(\tau^{k-1})$ for each $k\geq 1$"}.
\end{remark}
\par Now the following results are characterized the composition
operators $C_\tau$ on Orlicz spaces $L^\phi(\Omega)$ with finite
descent:
\begin{theorem}
If the map the map $\tau: R(\tau^N)\rightarrow R(\tau^N)$ is one-one
for some $N$, then descent of composition operator is less than or
equal to $N$.
\end{theorem}
\begin{proof}
Suppose that the map $\tau: R(\tau^N) \rightarrow R(\tau^N)$ is
one-to-one. Let $f\in R(C^N_\tau)$. Then $f=C^N_\tau(g)$ for some
$g\in L^\phi(\Omega)$. Now define the function
$$ h(x) = \begin{cases} g(y)~~;\,\,\,\,\mbox{if~~ $y\in R(\tau^N)$~~ and~~ $\tau(y)=x$ }\\
0~~;\,\,\,\,\,\,\,\,\,\,\,\mbox{otherwise}
\end{cases}$$
As $\tau :  R(\tau^N) \rightarrow R(\tau^N)$ is one-to-one, hence
the map $h$ is well defined and $g \in L^\phi(\Omega)$ implies that
$ h$ also in $L^\phi(\Omega)$. Note that $h\circ \tau (x)= g(x)$ for
all $x\in R(\tau^N)$. Now we have,
\begin{align*}
h(\tau^{n+1}(x)) & = h(\tau(\tau^N(x))\\
& = g(\tau^N(x))\\
& = f(x)
\end{align*}
This shows that $ f\in R(C^{N+1}_\tau)$. Therefore,
$R(C^N_\tau)\subseteq R(C^{N+1}_\tau)$. Hence this implies that
descent of the composition operator is less than or equal to $N$
\end{proof}
\begin{corollary}
 If descent of composition operator is infinite then the map $\tau:
R(\tau^k)\rightarrow R(\tau^k)$ is not one-one for all $k\geq 0$.
\end{corollary}
\begin{theorem}
Assume that in the measure space $\Omega= (\Omega, \Sigma, \mu)$ every singleton set has positive measure. Then
descent of composition operator is infinite if the map $\tau:
R(\tau^k)\rightarrow R(\tau^k)$ is not one-one for all $k\geq 0$.
\end{theorem}
\begin{proof}
Suppose that the map $\tau$ is not one-one. Then there exist
$x_1\neq x_2 \in R(\tau^k)$ such that $\tau(x_1)=\tau(x_2)$. Take
$\Omega_1= \{x\in \Omega: \tau^k(x)=x_1\} $ and $\Omega_2= \{x\in
\Omega: \tau^k(x)=x_2\}$. Note that $\Omega_1\cap \Omega_2=
\emptyset$ and $\tau^{k+1}(\Omega_1)=\tau^{k+1}(\Omega_2)$. Now
consider the function $f=\chi_{\Omega_1} - \chi_{\Omega_2}$. As
$0<\mu(\Omega_1)<\infty $ and $0<\mu(\Omega_2)<\infty$, hence
$\chi_{\Omega_1}$ and $\chi_{\Omega_2}$ are belongs to Orlicz space
$L^\phi \Omega)$ and hence $f$ also in $L^\phi(\Omega)$. Take $g =
\chi_{\{x_1\}} - \chi_{\{x_2\}}$. Then $g$ also in $L^\phi(\Omega)$ and
$C^k_\tau(g)= f$. This implies that $f\in R(C^k_\tau)$. Now claim is
that $f\notin R(C^{k+1}_\tau)$. If so then $f=C^{k+1}_\tau(g)$ for
some $g\in L^\phi(\Omega)$. Then
\begin{align*}
1= f(\Omega_1) & = C^{k+1}_\tau g(\Omega_1)\\
& = g(\tau^{k+1}(\Omega_1))\\
& = g(\tau^{k+1}(\Omega_2))\\
& = f(\Omega_2)\\
& = -1,
\end{align*}
which is a contradiction. Hence we have $f\in R(C^k_\tau)$ but
$f\notin R(C^{k+1}_\tau)$. Since $k$ is arbitrary, hence descent of
$C_\tau$ is infinite.
\end{proof}
\begin{corollary}
Suppose that the measure $\mu$ is as above. Then descent of the
composition operator is $N$ if and only if there exists a natural
number $N$ (smallest) such that the map $\tau: R(\tau^N)\rightarrow
R(\tau^N)$ is one-one.
\end{corollary}
\section*{Acknowledgement}
 One of the authors (R. Kr. Giri) thanks the financial assistantship received from
 the Ministry of Human Resource Development
 (M.H.R.D.), Govt. of India.
\bibliographystyle{amsplain}

\end{document}